\documentclass[11pt]{article}
\usepackage{amsmath}
\usepackage{amssymb}
\usepackage{graphicx}
\usepackage{amsthm}
\usepackage{verbatim}
\usepackage{color}
\usepackage{enumerate}
\usepackage{amsfonts}
\usepackage[all]{xy}
\usepackage{fancyhdr}
\usepackage{enumerate}
\theoremstyle{plain}

\newtheorem{theorem}{Theorem}[subsection]

\newtheorem{proposition}[theorem]{Proposition}
\newtheorem{corollary}[theorem]{Corollary}

\theoremstyle{remark}
\newtheorem{remark}[theorem]{Remark}

\theoremstyle{definition}
\newtheorem{definition}[theorem]{Definition}

\begin{document}
\title{On a new geometric homology theory and an application in categorical Gromov-Witten theory}

\author{Hao Yu}



\maketitle

\begin{abstract}
The purpose of this paper is twofold: 1. we prove the triangulability of smooth orbifolds with corners, generalizing the same statement for orbifolds. 2. based on 1, we propose a new homology theory. We call it geometric homology theory (GHT for abbreviaty). GHT is a natural and flexible generalization of singular homology. It has some advantages overcoming the unpleasant combinatoric rigidity of singular homology, e.g. ill-defined pullbacks of singular chains along fiber bundles.  The method we use are mainly based on the celebrated stratification and triangulation theories of Lie groupoids and their orbit spaces, as well as the extension to Lie groupoids with corners by us. We illustrate a simple application of GHT in categorical Gromov-Witten theory, initiatied by Costello. We will develop further of this theory in our sequel paper.


\end{abstract}
\medskip
\tableofcontents

\section{Introduction}
In this paper, we introduce a new homology theory, we call it \textit{geometric homology theory}, or GHT for brevity.  This theory is based on the notion of geometric chains, first introduced by Voronov et al. (\cite{vhz}). GHT has several nice properties including (but not limited to): first, the homology groups in GHT are isomorphic to singular homology groups; Secondly, chains in GHT behave quite well with some natural operations including Cartesian products and pull backs along fiber bundles. This is in sharp contrast with the singular chains, where they do not have canonical definitions when acted by these operations. The third, the fundamental chains of manifolds (or orbifolds) with corners can be canonically defined in GHT, while there is certainly no such for singular chains. Due to these nice properties, we believe it will be useful in some cases if one uses geometric chains in place of singular chains.

Specifically, we define the (singular) geometric chains first (see section 3 for detailed notions).
\begin{definition}\label{main_def}
(singular geometric chains).
Let $P$ be any compact connected oriented (smooth) orbifold with corners of dimension, say $n$.  A \textit{geometric simplex} of dimension $n$ is the underlying space $|P|$  of $P$ (with induced orientation from $P$) together with its induced corner structure, that is, $\partial_k |P| := |\partial_k P|$ for each $k\ge 0$ for each $k\ge 0$. We also call these $\partial_k |P|$ the $k$-dimensional faces of $|P|$. We define the  boundary  $\partial |P| := \sum_{i}\partial^i_{n-1}|P|$. It means that we forget the inherent orbifoldic structure of $P$ except for preserving their corners relations. 
Given a topological space $X$,  a \textit{singular geometric simplex} is a pair $(Q, f)$ where $Q$ is a geometric simplex and $f: P \rightarrow X$ is a continuous map from $P$ to $X$. The dimension of $Q$ is called the dimension of the singular geometric simplex $(Q, f)$.  A \textit{singular geometric chain} is a formal linear combination over $Q$ of singular geometric simplices $(Q, f)$, modulo the equivalence relation induced from the following relations (\label{eqiv}):
\begin{itemize}
\item $(Q, f) = -(-Q, f)$ \\
\item $(Q_1, f_1) + (Q_2, f_2) = (Q_1 \cup Q_2, f_1 \cup f_2)$
\item for $(Q_1, f_{q_1}), (Q_2, f_{q_2})$ and $Q(P, f_P)$, if $|P|$ is the gluing (in the usual sense) of $|Q_1|$ and $|Q_2|$ along their lower dimensional faces compatible with $f_p, f_{q_1}, f_{q_2}$, then $(P, f_p) = (Q_1, f_{q_1}) + (Q_2, f_{q_2})$
\end{itemize}
where $-Q$ is $Q$ with reversed orientation. Note that the above relations themselves are not equivalence relations (the third relation is not closed under transition for example), so our equivalence relation is induced from them. The boundary operator of singular geometric chains is defined as usual.

\end{definition}
We denote the set of singular geometric chains of dimension $k$ by $GC_k(X, Q)$. 

In the following if there is no ambiguity we call singular geometric chain (geometric simplex resp.) geometric chain (geometric simplex resp.) for simplicity. 
\begin{remark}
As we are working with orbifolds, we use coefficient ${Q}$ instead of ${Z}$. The reader can also check paragraph in \cite{vhz} to understand why using ${Q}$ is necessary.
\end{remark}

It is very crucial to emphasize that, due to the topological nature of geometric simplex as well as the equivalence relations, our theory does \textit{not} deal with cobordism of manifolds with corners (\cite{ce,g}) or cobordism of orbifolds with corners. This is the most subtlety. Our theory is to levarage the sturcture of orbifolds with corners to allow more complex and flexible goemetric objects than simplexes, and do what singular simplicial homology does on the \textit{underlying topological space with corners} of orbifolds with corners, not to take the full structure of orbifolds with corners into consideration.

Let us give a remark. In ordinary homology, the addition operator of the chain group is simply \textit{formal}, lacking some geometric meanings, which is a little unsatisfactory. Intuitively, the geometric meaning of addition is the union of chains. Under this explanation, then, it is true that some well positioned geometric chains can glue to form a larger new geometric chain, thus the latter is manifestly the addition of the former. This is the motivation to  propse the above definition of geometric chains. We notice that if we impose the similar equivalence relations on singular chains, the resulting homology is isomorphic to singular homology. This is due to the well known fact that a finer triangulation of a closed simplicial complex will not change its homology class.

We will prove the following theorem.
\begin{theorem}\label{main_thm}
(geometric homology is isomorphic to singular homology).
The set of geometric chains of $X$  forms a complex, still denoted by $GC_*(X,Q)$ over $Q$, graded by the dimensions of chains, with the boundary of a chain $f$ given by $(\partial P, f|\partial P )$,
where $\partial P$ is the sum of codimension one faces of $P$ with the induced orientation. The homology groups of $GC_*(X,Q)$ are denoted by $GH_*(X,Q)$. Then we have
\begin{equation}
GH_*(X,Q) \cong H_*(X,Q)
\end{equation}
\end{theorem}
In particular, if $P$ is restricted to be a manifold with corners, we get a sub-complex of geometric chain complex. We denote it by $GC^m_*(X,Q)$.
\begin{corollary}\label{main_cor}
The homology groups of $GC^m_*(X,Q)$ are isomorphic to singular homology groups, i.e.
\begin{equation}
GH^m_*(X,Q) \cong H_*(X,Q)
\end{equation}
\end{corollary}
This corollary is perhaps more interesting for some readers.
After completion of the ideas and methodologies used in the paper, we are aware that this corollary has previously appeared in \cite{cd} with the proof also being based on triangulations theory. And we then partially refer to their definition of chains in our paper. However, no any detail is provided there except for pointing out some references.  To me, strictly speaking, the result there should also be formulated in topological form as we did above, which is due to the reason that triangulation is only of topological nature but smooth manifolds is known that they may have different smooth structures corresponding to the same topology. The usefulness of this kind of new homology theory is already clearly shown in their paper (\cite{cd}), which we are delighted to see. We remark that the triangulation theory of differentiable manifolds (or with boundary) has already been a classical topic in topology in its early era, and its development is a long history.  One direction of generalization of this theory is to smooth closed orbifolds (\cite{ppt}),  this is already completely nontrivial. This paper takes a further step towards smooth orbifolds with corners that is also nontrivial. The theorem \ref{main_thm} is somewhat surprising. In \cite{vhz}, Voronov et al. made a remark on this property. We quote their words here "Our notion of chains leads
to a version of oriented bordism theory via passing to homology. If we impose extra
equivalence relations, such as some kind of a suspension isomorphism, as in \cite{ja},
or work with piecewise smooth geometric chains and treat them as currents, in the
spirit of \cite{fooo}, we may obtain a complex whose homology is isomorphic to the
ordinary real homology of X". Although they didn't state it as a formal conjecture, we nevertheless take the liberty to regard this as an implicit conjecture. Thus from this view, our theorem \ref{main_thm} answers their question affirmably, however no additional constrains needed.

Similarly, one can define the geometric cohomology groups, as well as cup products, cap products, and so on. Essentially, properties hold in singular homology (cohomology) theory should still hold in GHT. This observation can be deduced from the very method of the proof of the main theorem \ref{main_thm}. We note that it is a general principle, one still needs to prove each basic property one by one. As the first of a potential series papers in developing GHT, we decide not to introduce geometric cohomology theory here, leaving it to our subsequent work.

For applications, we mention that (which is also a motivation of developing GHT) in \cite{cos04,cos05} Costello constructed an algebraic counterpart of the theory of Mirror symmetry, in particular, the construction of Gromov-Witten potential in pure algebra manner. In order to deal with pullbacks of chains, he has to consider $S$-equivariant chains and $S^1$ homotopy coinvariants, which is due to the singular chains not being able to canonically defined under pullbacks as mentioned above. Also, he has to use $S^1$ equivariant chains to define fundamental chains for orbifolds with boundary. More specifically, he needs to construct a Batalin-Vilkovsky algebraic structure on the singular chain complexes on the compactification of moduli spaces of Rienmann surface with marked points, the well-known Deligne-Mumford spaces, and moreover needs to use fundamental chains as solutions of quantum master equations that encodes the fundamental classes of Deligne-Mumford spaces and thus encodes Gromov-Witten potentials. By using geometric chains, we give an extension of his results to open-closed topological conformal theory and prove the similar results (see \cite{hy}). This is also related to the typical motivation of geometric chains theory as "we need the unit circle $S^1$ to have a canonical fundamental cycle and the moduli spaces which we consider to have canonical fundamental chains, irrespective of the choice of a triangulation. On the other hand, we believe it is generally better to work at the more basic level of chains rather than that of homology", quoted from \cite{vhz}. We do not need to use equivariant homotopy theory in this case since we are using geometric chains instead of singular chains, everything there looks very natural. Other potential applications will be studied as a future work.

We briefly sketch the ideas of the proof as follows. Regarding orbifolds as proper Lie groupoids $s,t: G \Rightarrow M$, with discrete isotropy groups is the modern views of orbifolds (\cite{mo}). The method of proof relies heavily on the celebrated stratifications and triangulations theory of Lie groupoids. 
Any proper Lie groupoid has a canonical stratification given by the connected components of Morita equivalence decomposition. Furthermore, the stratification is a Whitney stratification. The same are also true when passing to the orbit space. The result of Thom and Mather in \cite{mat70} shows that any space that has a Whitney-stratification admits an abstract pre-stratification structure. And the result of \cite{v} shows that if a space has an abstract pre-stratification structure, then it has a triangulation compatible with the abstract pre-stratification. In particular, the orbit space has an abstract pre-stratification structure and thus has a triangulation. The innovative idea of ours is to introduce the notion of Lie groupoid with corners and prove analogous results as for Lie groupoids. One of them is the following extension of Bierstone's theorem (\cite{b75}) to the case of manifolds with corners: the orbit space of an Euclidean space with corners acted by a compact Lie group admits a canonical stratification that is Whitney stratification, and it coincides with the stratification given by orbit type. This stratification also respects the corner structure. With these extensions, the above machinery then all work, and we can thus prove the theorem \ref{main_thm}.

The organization of the remaining of this paper is as follows. In section 3, we review relevant concepts and properties on orbifolds, orbifolds with corners and Lie groupoids and introduce new notions of Lie groupoids with corners. We prove an equivalence between orbifolds with corners and proper Lie groupoids with corners by studying a local model of Lie groupoids with corners. Section 4 discusses the celebrated stratification theory on proper Lie groupoids and presents generalization to proper Lie groupoids with corners. 
We then prove our main results in section 5.  An application in categorical Gromov-Witten theory is briefly discussed in section 6. 

\section{Notation}
Let $R^n$ denote the $n$-dimensional Euclidean space. $R^n_+:=\{(x_1,x_2,\dots,x_n) | x_i \geq 0, i=1,\dots,n\}$ denotes the subset of $R^n$ consisting of points whose coordinates are all nonnegative reals. And let $R^n_k:=\{(x_1,x_2,\dots,x_n) | x_i\geq 0, i=1,\dots,k\}$ denote the subset of $R^n$ with the first $k$ coordinates non-negative. And let $R^n_{+k}$ denote a subset of $R^n_+$ consisting of points whose coordinates have exact $n-k$ 0.

We need to pay attention to the definition of smooth maps between smooth manifolds with corners. There are several different definitions available in literature (\cite{ce, me, mon, jo}). As for a smooth map, being a weak smooth map seems to be a must. Those alternative definitions differ in that they put additional constraints, especially restrict the behavior of smooth maps on the boundaries (or corners) of the manifold with corners. We heretoforth adopt the definition of  category of smooth manifolds with corners by Dominic Joyce (\cite{jo}). We highly refer the readers to that paper for the full account of details about relevant definitions, e.g., weak smooth maps, submersions, etc. \textit{Having said that, however, we need to make it clear to the readers that we actually \textit{never} need a general definition of smooth map between manifolds with corners in the current work}. Where we need such is in the definitions of orbifolds with corners and Lie groupoids with corners, but they are very special cases rather than general ones, and the definitions we take should agree on \textit{any} notions of smooth maps mentioned above when restricting to these cases.


For any smooth manifold with corners $M$,  we denote by $Bord_k({M})$ a subset of ${M}$ consisting of points each of which has a neighbourhood diffemorphic to an open set in $R^n_{k}$, sending that point to 0. It is called the \textit{$k$-corner} of $M$. Depending on the context, if there is no ambiguity, we also call \textit{connected component} of it a $k$-corner. Each connected component of $Bord_k(M)$ is a smooth manifold (with the induced topology) of which the closure is a manifold with corners (the latter is the notion of $k$-\textit{corners} defined in \cite{jo}). We call each closure of such a connected component a \textit{$k$-face} of $M$.  The \textit{boundary} of an oriented smooth manifold with corners of dimension $n$, denoted as $\partial M$, is the sum of its $n-1$ faces with induced orientations. 
$M$ together with all $Bord_k(M)$ is called the \textit{corner structure} of $M$. When a Lie group $G$ acts continuously on a topological space $M$, its orbit space is denoted by $|M/G|$. A smooth action of a Lie group $G$ on a manifold with corners $M$ is a smooth map
\begin{equation}
f: G \times M \rightarrow M
\end{equation}
such that $f|_{G \times {Bord_k(M)}} \subseteq Bord_k{M}$.
Moreover, if $M=R^n_k$ and $G$ acts on each $Bord_l(M)$ lineally, we say $R^n_k$ is a \textit{linear corner representation} of $G$.

In the large part of our paper, we frequently use a technique in the arguments: expanding symmetrically along the first $k$ axes of $R^n_k$ to turn it into $R^n$ in a manifest way, and so turn the \textit{corner} case to the \textit{normal}  (no boundaries and corners) case. We call this technique \textit{doubling} argument.


\section{Orbifolds, orbifolds with corners and Lie groupoids}
We review the preliminaries on orbifolds, orbifolds with corners and Lie groupoids. The standard reference for orbifolds and Lie groupoids are \cite{alr,mo,mp} as well as the relevant work \cite{cm, mm}, see also the pioneering work \cite{s}.
We start discussing orbifolds by presenting its original definition, then we show that they can be equivalently reformulated as proper Lie groupoids with discrete isotropy groups.
\begin{definition}\label{orbifold}
(Orbifolds).
Let $X$ be a topological space, and fix $n \geq 0$.
\begin{itemize}
\item An $n$-dimensional orbifold chart on $X$ is given by a connected open subset
$\tilde{U}\subseteq R^n$, a finite group $G$ acting smoothly on $\tilde{U}$, 
and a map $\phi:\tilde{U}\rightarrow X$ so that $\phi$ is $G$-invariant and induces a homeomorphism of $\tilde{U}/G$  onto
an open subset $U \subseteq  X$.
\item An embedding $\lambda : (\tilde{U} , G, \varphi  ) \hookrightarrow (\tilde{V} ,H, \psi)$ between two such charts is a
smooth embedding $\lambda : \tilde{U} \hookrightarrow \tilde{V}$
 with $\psi \lambda = \varphi$.  An orbifold atlas on $X$ is a family $\mathcal{U} = {(\tilde{U} , G, \varphi)}$ of such charts, which
cover $X$ and are locally compatible: given any two charts $(\tilde{U} , G, \varphi )$ for
$U = \varphi(\tilde{U}) \subset X$ and $(\tilde{V},H,\psi  )$ for $V \subseteq X$, and a point $x \in U \cap V$ , there
exists an open neighbourhood $W \subseteq U \cap V$ of $x$ and a chart $(\tilde{W},K,\mu  )$ for $W$
such that there are embeddings $(\tilde{W},K,\mu  ) \hookrightarrow (\tilde{U}, G, \varphi  )$ and $(\tilde{W},K,\mu  )\hookrightarrow (\tilde{V},H,\psi)$.
\item An atlas $\mathcal{U}$ is said to refine another atlas $\mathcal{V}$ if for every chart in $\mathcal{U}$ there
exists an embedding into some chart of $\mathcal{V}$. Two orbifold atlases are said to be
equivalent if they have a common refinement.
\end{itemize}
An orbifold $X$ of dimension $n$ or an $n$-orbifold is a paracompact Hausdorff space $X$ equipped with an equivalence class $\mathcal{U}$ of $n$-dimensional orbifold atlases.
\end{definition}
As geometric chains are built on parts of orbifolds with corners, we need to similarly extend the definition of \ref{orbifold}, as what we did from manifolds to manifolds with corners.
\begin{definition}\label{orbi_corners}
(orbifolds with corners).
Let $X$ be a topological space, and fix $n \geq 0$.
\begin{itemize}
\item A generalized $n$-dimensional orbifold chart on $X$ is given by a connected open subset
$\tilde{U} \subseteq R^n_+$, a finite group $G$ acting smoothly on $\tilde{U}$, 
and a map $\varphi :\tilde{U} \rightarrow X$ so that $\varphi$ is $G$-invariant and if for any $0\leq k \leq n$ $\tilde{U}_k := \tilde{U} \cap R^n_{+k} \neq \emptyset $ then $\phi$ restricted to any connected component of ${\tilde{U}_k}$ is $G$-invariant, and $\phi$ induces a homeomorphism of $\tilde{U}/G$  onto
an open subset $U \subseteq X$.
\item An embedding $\lambda : (\tilde{U} , G, \varphi  ) \hookrightarrow (\tilde{V} ,H,\psi  )$ between two such charts is a
collection of smooth embeddings $\lambda_k : \tilde{U}_k \hookrightarrow \tilde{V}_k$.
 with $\psi \lambda = \varphi$.  An orbifold atlas on $X$ is a family $\mathcal{U} = {(\tilde{U} , G, \varphi  )}$ of such charts, which
cover $X$ and are locally compatible: given any two charts $(\tilde{U} , G, \varphi  )$ for
$U = \varphi(\tilde{U}) \subseteq X$ and $(\tilde{V} ,H, \phi )$ for $V \subseteq X$ and a point $x \in U \cap V$ , there
exists an open neighbourhood $W \subseteq U \cap V$ of $x$ and a chart $(\tilde{W},K,\mu  )$ for $W$
such that there are embeddings $(\tilde{W},K,\mu  ) \hookrightarrow (\tilde{U} , G, \varphi  )$ and $(\tilde{W},K,\mu  ) \hookrightarrow (\tilde{V} ,H,\psi  )$.
\item An atlas $\mathcal{U}$ is said to refine another atlas $\mathcal{V}$ if for every chart in $\mathcal{U}$ there
exists an embedding into some chart of $\mathcal{V}$. Two generalized orbifold atlases are said to be
equivalent if they have a common refinement.
\end{itemize}
An orbifold with corners $X$ of dimension $n$ or an $n$-orbifold with corners is a paracompact Hausdorff space $X$ equipped with an equivalence class $\mathcal{U}$ of $n$-dimensional generalized orbifold atlases.
\end{definition}
\begin{remark}
It is possible to give a notion of orbifolds with corners for which the first condition is modified to allow $G$ not necessarilly keep each connected component of $k$-corner invariant but only keep the whole $k$-corner invariant. However, actually these two definitions are equivalent. We will leave the proof to the readers. 
\end{remark}
From the definition we see that the notion of orbifolds are tied with the notion of spaces with finite group action, and more generally smooth spaces with Lie groups actions. A generalization of the latter is the notion of Lie groupoid.
\begin{definition}
(Lie groupoid).
A Lie groupoid consists of the following data: two smooth manifolds $\mathcal{G}$ and ${M}$, called the space of sources and space of objects, two smooth maps $s,t:\mathcal{G} \rightarrow \mathcal{M}$, called the source and target maps, such that $s$ or $t$ is a surjective submersion, a partial defined smooth multiplication $m: \mathcal{G}^{(2)}\rightarrow \mathcal{G}$, defined on the space of composable arrows $\mathcal{G}^{(2)} = \{(g,h)\in \mathcal{G} | s(g) = t(h)\}$ (by being the surjective submersion of $s$ or $t$, this set is a smooth manifold), a unit section $u:  {M} \rightarrow \mathcal{G}$, and an inversion $i:\mathcal{G} \rightarrow \mathcal{G}$, satisfying group-like axioms. We denote the above Lie groupoid by $\mathcal{G} \Rightarrow M$.

\end{definition}

We assume that the readers have some familiarities with basic notions and properties of Lie groupoids. For comprehensive introduction we refer the book by Mackenziesee (\cite{mac}).

As conventional, for an open set $U\subseteq M$ we denote by $\mathcal{G}_U$ the restricted Lie groupoid of $\mathcal{G}$ to $U$. The isotropy group of $\mathcal{G}$ at $x$, in particular, is denoted by $\mathcal{G}_x$. And $O_x$ denotes the orbit of $x\in M$ under the canonical action of $\mathcal{G}$ on $M$.

In \cite{mo} it established a very important and nice property that an orbifold is equivalently described as a proper Lie groupoid with discrete isotropy groups (and also equivalently formulated as a proper Lie groupoid with etale isotropy groups if considering Morita equivalence). To work in the category of orbifolds with corners, we introduce a new notion of Lie groupoids with corners.
\begin{definition}
(Lie groupoid with corners).
A Lie groupoid with corners consists of the following data: two smooth manifolds with corners $\mathcal{G}$ and ${M}$, and all the data as in the definition of Lie groupoids, such that the source and target maps respect the corner structures of $\mathcal{G}$ and ${M}$, i.e.  for each $k \ge 0$, the inverse image of each connected component of $Bord_k({M})$ is a union of several connected components of $Bord_l(\mathcal{G})$ for a possible list of $l$s,  and the image of each connected component of $Bord_l(\mathcal{G})$ lies in one connected component of $Bord_k(M)$. 
\end{definition}
One can deduce then that $m, u, i$ all respect the corners structures of $\mathcal{G}$ and ${M}$.

\begin{remark}
The requirement that the source and target maps are respect the corner sturctures is motivated by the  local model of orbifolds with corners. Without this condition, the equivalence of orbifolds with corners and Lie groupoids with corners discussed below should not be true.
\end{remark}
\begin{remark}
One can give a candidate of definition of smooth map between manifolds with corners, by asking that in addition to being a weak smooth map it should also be continuous when restricting to corners of source and target manifolds with corners. We don't plan to present a detailed definition here, but just mention that this definition of smooth map implies that it is respect the corner structures as aforedmentioned. In comparison, the definition in \cite{jo} only cares on codimention one corners (i.e. boundaries). We plan to pursue this direction of research in the future.
\end{remark}
We often write $\mathcal{G} \ltimes M$ for a Lie groupoid (resp. with corners) associated to a Lie group action on a manifold (resp. with corners) $M$.

Naturally, one expect there is also an equivalence between orbifolds with corners and proper Lie groupoids with corners and discrete isotropy groups. For proving this, we need the following local model of proper Lie groupoids with corners around a point.
\begin{proposition}\label{local model of proper Lie groupoids with corners about a point}
(local model of proper Lie groupoids with corners about a point).
Let $\mathcal{G} \Rightarrow M$ be a proper Lie groupoid with corners over $M$. And let $x\in M$ be a point. There is a neighbourhood $U$ of $x$ in $M$, diffemorphic to $O \times W$, where $O \subseteq O_x$ is an open set, $W \subseteq N_x$ is an $\mathcal{G}_x$ invariant open set in the normal space $N_x$ to $O_x$ at $x$, such that under this diffemorphism $\mathcal{G}_U$ is isomorphic to the product of the pair groupoid $O \times O \Rightarrow O$ with $\mathcal{G}_x \ltimes W$.
\end{proposition}
\begin{proof}
W.l.o.g assume $s$ is a surjective submersion. For any $x \in M$, choose any $g\in \mathcal{G}_x$, we can then find neighbourhoods $U$ around $g$ and $V$ around $x$, such that on $U$ $s$ is given by projection onto the second factor:
\begin{equation}
s|_U: R^l_m \times R^n_k \rightarrow R^n_k
\end{equation}
for some $l, m$ and $n, k$.
$\mathcal{G}_V$ is an non-empty open set of $U$ since it contains $g$. Now we use "doubling" argument to extend manifolds with corners $\mathcal{G}_V$ and $V$ to manifolds without corners. The idea is simply to symmetrically expand $R^n_k$ along first $k$ coordinate directions to make it into $R^n$, so that $\mathcal{G}_V$ will expand to be an open set $\tilde{\mathcal{G}}:=Sym(\mathcal{G}_V)$ in $R^{l+n}$, and $V$ becomes an open set $Sym(V)$ in $R^n$. The $s,t$ maps on $\tilde{\mathcal{G}}$ are defined to be the natural extension of $s,t$ on $\mathcal{G}_V$ that become invariant under this symmetry. $\tilde{\mathcal{G}} \Rightarrow Sym(V)$ is a proper Lie groupoid. By a local model for proper Lie groupoids (cf. \cite{cm} prop. 2.30) around $x$ we know that there is a neighbourhood $\mathcal{\tilde{N}}$ around $x$ in $Sym(V)$ such that when restricting to $\mathcal{\tilde{N}}$, there is a $\tilde{\mathcal{G}}_x$ invariant open set $\tilde{W}$ in $\tilde{N}_x$ (the normal space in $Sym(V)$ to the orbit $O_x$) and an open set $O$ in $O_x \cap Sym(V)$, such that  $\tilde{\mathcal{G}}_\mathcal{\tilde{N}}$ is isomorphic to the product of linear action groupoids $\tilde{\mathcal{G}}_x \ltimes \tilde{W}$ with pair groupoid $O \times O\Rightarrow O$. So when restricting back to the original Lie groupoid with corners $\mathcal{G}_V$ and $V$, we see that there is a neighbourhood $\mathcal{N} \subseteq V$ around $x$, an $(\mathcal{G}_V)_x$ invariant open set $W$ in $N_x$ (the normal space in $V$ to the orbit $O_x$, which is $R^n_k$) and an open set $O$ in $O_x\cap V$ so that $({\mathcal{G}_V})_\mathcal{N}$ is isomorphic to the product of a linear action groupoid with corners $(\mathcal{G}_V)_x \ltimes
W$ with product groupoid $O \times O \Rightarrow O$. This is the local model of Lie groupoid with corners. Also $(\mathcal{G}_V)_\mathcal{N}$ is Morita equivalence to $(\mathcal{G}_V)_x \ltimes W$. Moreover, $(\mathcal{G}_V)_\mathcal{N}$ is also Morita equivalence to $(\mathcal{G}_V)_x \ltimes N_x$. For this we use also that $N_x\cong R^n_k$ admits arbitrarily small $(\mathcal{G}_V)_x$-invariant open neighbourhoods of the origin which are equivariantly diffemorphic to $N_x$. The latter can be deduced from the corresponding statement for $R^n$ by using doubling argument straightforwardly.
\end{proof}
\begin{theorem}\label{equiv_orbi_groupoid_corners}
There is an equivalence between orbifolds with corners defined in the sense of \ref{orbi_corners} and proper Lie groupoids with corners with discrete isotropy groups.
\end{theorem}
\begin{proof}
$\Longrightarrow$ If $X$ is given by the original definition \ref{orbi_corners}, locally, $X$ is given by the action groupoid $\phi \ltimes U$, i.e. locally $X$ is homeomorphic to the orbit space $|U/\phi|$. We can glue such action groupoids together to form a (proper) Lie groupoid with corners whose isotropy groups are all finite groups.

$\Longleftarrow$ If $X$ is homeomorphic to the orbit space of a proper Lie groupoid with corners with discrete isotropy groups, then by the proposition \ref{local model of proper Lie groupoids with corners about a point}, we know that locally $\mathcal{G}$ is Morita equivalent to a linear action groupoid $\mathcal{G}_x \ltimes U_x$, for an invariant open set  $U_x\subseteq N_x$.
As $\mathcal{G}_x$ is discrete and compact, it is a finite group. Thus locally $X$ is homeomorphic to $|U_x/\mathcal{G}_x|$. These $G_x \ltimes U_x$ give the generalized orbifold charts of $X$ in the sense of \ref{orbi_corners}.
We thus conclude the proof of the theorem \ref{equiv_orbi_groupoid_corners}.
\end{proof}
One can also prove an analogue of \textit{Linearization theorem} for Lie groupoids with corners, along with the lines of using Zung's theorem (\cite{z}) (for orbit a single point) and Weinstein's trick (\cite{w}). 
As we do not need it in this paper, and the proposition \ref{local model of proper Lie groupoids with corners about a point} above is enough for our main result, theorem \ref{main_thm}, we will prove it in our sequel paper.

A somewhat surprising property of orbifolds is that it can always be described as a global quotient.
\begin{proposition}\label{global_quotient}
Every classical $n$-orbifold $X$ is diffemorphic to a quotient
orbifold for a smooth, effective, and almost free $O(n)$-action on a frame bundle (which is a smooth
manifold) $Fr(X)$.
\end{proposition}
We refer the readers to \cite{alr} for details on this result.
\section{Stratification theory on proper Lie groupoids (with corners)}
In this section, we discuss the stratification and triangulation theories developed particularly for orbifolds, and more generally for orbispaces. It was generalized to Lie groupoids setting, leading to a unified and more compact theory.

For stratified spaces we understand they are Hausdorff second-countable topological spaces $X$ endowed with a locally finite partition $\mathcal{S} = \{X_i|i\in I\}$ such that 
\begin{enumerate}[(1)]
\item Each $X_i$ endowed with the subspace topology, is a locally closed, connected subspace of $X$ and a smooth manifold.
\item (locally finiteness) every point $x \in X$ has a neighbourhood that intersects finite members of $\mathcal{S}$.
\item (frontier condition) the closure of $X_i$ is the union of $X_i$ with the members of $\mathcal{S}$ of strict lower dimension.
\end{enumerate}

If $X$ has a differentiable structure, for example a differential manifold (or more generally a differentiable space developed in \cite{ns}), one usually further requires it being a Whitney stratification (cf. \cite{cm}).
\begin{definition}
(Whitney stratification).
Let $M$ be a smooth manifold (without boundary). Let $N$ be a subset of $M$. Let $\mathcal{S}$ be a stratification of $N$. Then $\mathcal{S}$ is said to be a \textit{Whitney stratification} if the following conditions hold
\begin{enumerate}[(A)]
\item For any strata $R, S \in \mathcal{S}$, given
any sequence $\{x_i\}$ of points in $R$ such that $x_i \rightarrow y \in S$ and $TX_{x_i}
$ converges to some $r$-plane ($r$=dim($R$)) $\tau\subset TM_y$, we have $TY_y \subset \tau$.

\item  Let
$\{x_i\}$ be a sequence of points in $R$, converging to $y$ and $\{y_i\}$ a sequence of points in $Y$, also
converging to $y$. Suppose $TX_{x_i}$
 converges to some $r$-plane $\tau \subset R^n$
and that $x_i \neq y_i$ for all $i$ and the secants $(x_i y_i)$ converge (in projective space $P^{n-1}$) to some line $l\subset R^n$
. Then $l \subset \tau$.
\end{enumerate}
When $M$ is a manifold with corners, one needs some modifications as follows: 
\begin{enumerate}\label{str_mani_cor}
\item (for condition (1)). We further require that the partition $\mathcal{S} = \{X_i|i\in I\}$ should be compatible with the corner strcture of $X$, that is, the restriction of $\mathcal{S}$ on any $k$-face $F$ of $X$ itself forms a stratification of $F$.
\item (for condition (A) of Whitney stratification). We require that when $S$ is on a corner of $R$, then $T_x(M)$ should be understood as the $R$-dimensional vector space tangent to $x$
on the doubling of the local model $R^n_k$ around $x$ which is $R^n$.
\end{enumerate}
\end{definition}
A closely related notion of \textit{abstract pre-stratification} is a very essential concept in the stratification and triangulation theory(cf. \cite{mat70, v}).

Any Lie groupoid has a canonical stratification induced by the Morita type equivalence classes, denoted by $x \sim_{\mathcal{M}} y$, i.e. $x \sim_{\mathcal{M}} y \Longleftrightarrow (\mathcal{G}_x, N_x) \cong (\mathcal{G}_y, N_y)$ where $N_x, N_y$ are normal (normal spaces to the orbits $O_x$ and $O_y$ respectively) representations of $\mathcal{G}_x, \mathcal{G}_y$, respectively. Since Morita type equivalence relation is invariant on any orbit, it also induces a stratification on the orbit space. We refer the readers to \cite{cm} for details on Morita type equivalence and canonical stratification. 


There are other equivalence relations, for example, \textit{isotropy isomorphism equivalence} that is defined by: $x \cong y \Leftrightarrow \mathcal{G}_x \cong \mathcal{G}_y$. However when passing to connected components of each equivalence class they all give the same stratification.

If the Lie groupoid happens to be an action groupoid, then the stratification also agrees with the stratification induced by some equivalence relations specific to group actions. For example, \textit{orbit type equivalence} that is defined by $x \sim y \Leftrightarrow \mathcal{G}_x \sim \mathcal{G}_y$ (i.e. $\mathcal{G}_x$ and $\mathcal{G}_y$ are conjugate in $\mathcal{G}$), or isotropy isomorphism equivalence. We refer the interested readers to \cite{cm} for more details.

The key result in the theory on the stratification of proper Lie groupoid is that the canonical stratification on $M$ and $X$ are both Whitney stratifications (cf. \cite{cm} prop. 5.7), which we will generalize to Lie groupoids with corners.
\begin{theorem}\label{Whitney_stratification_proper_groupoid}
Let $\mathcal{G} \Rightarrow M$ be a proper Lie groupoid over $M$, 
then the canonical stratifications of $M$ and $X$ are Whitney stratifications.

Any Morita equivalence between two proper Lie groupoids induces an isomorphism of differetiable stratified space between their orbit spaces. 

\end{theorem}
The most innovative ingredient of this paper is to generalize the result above to Lie groupoids with corners. We need to talk first on what the Morita equivalence relation means in this setting.
\begin{definition}\label{morita_corners}(Morita equivalence for Lie groupoids with corners).
Let $\mathcal{G}$ be a Lie groupoid with corners over $M$ and $\mathcal{H}$ a Lie groupoid with corners over $N$. A morita equivalence between $\mathcal{G}$ and $\mathcal{H}$ is given by a principle $\mathcal{G}-\mathcal{H}$ bi-bundle, i.e. a manifold with corners $P$ endowed with
\begin{enumerate}
\item Surjective submersion $\alpha: P \rightarrow N, \beta: P \rightarrow M$ that compatible with corner structures of $P, M$ and $N$.
\item A left action of $\mathcal{H}$ on $P$ along the fibers of $\alpha$, compatible with the corner structure of $\mathcal{H}$ and $P$, which makes $\beta: P \rightarrow M$ into a principle $\mathcal{H}$-bundle.
\item A right action of $\mathcal{G}$ on $Q$ along the fibers of $\beta$, compatible with the corner structure of $\mathcal{G}$ and $Q$, which makes $\alpha: Q \rightarrow N$ into a principle $\mathcal{G}$-bundle.
\item the left and right action commute.
\end{enumerate}
$\mathcal{G}$ and $\mathcal{H}$ are Morita equivalence if such a bi-bundle exists.
\end{definition}
The following statement is the "corner" case extension of a standard result on the stratification of proper action groupoids (cf. \cite{cm} thm. 4.30).
\begin{proposition}\label{stratification_proper_groupoid_corner}
Let $\mathcal{G} \Rightarrow M$ be a proper Lie groupoid with corners over $M$. The canonical decompositions of $M$ and $X:=|M/\mathcal{G}|$ by the Morita type equivalence classes on each $k$-corner of $M$ and $X$ are stratifications. Moreover, this stratification on $M$ is a Whitney stratification.

Given a Morita equivalence between two Lie groupoids with corners $\mathcal{G}$ and $\mathcal{H}$, the induced homeomorphism at the level of orbit space preserves the canonical stratification.
\end{proposition}
\begin{proof}
For proper Lie groupoids over manifolds, this property is shown to be true in \cite{cm}. One can prove the "corner" case as follows: 

The manifold condition on $M$ (resp. $X$) are obvious satisfied, since each strata on each $k$-corner of $M$ (resp. $X$) is a smooth manifold. For locally finiteness condition, let $x \in M$. By the local model of Lie groupoids with corners \ref{local model of proper Lie groupoids with corners about a point}, locally there are neighbourhoods $U$ of $x$ and $W$ of $N_x$ such that the conclusion in \ref{local model of proper Lie groupoids with corners about a point} holds. Then the strata that intersects with $U$ corresponds to strata on $\mathcal{G}_x \ltimes W \subseteq N_x \cong R^n_k$. Using the doubling argument, we expand along $n-k$ axes in $R^n_k$ to turn the latter symmetrically into $R^n$, with the action of $\mathcal{G}_x \times Z_2^{n-k}$ of which the second factor acts by reflection along corresponding axis, i.e., $(g, \tau)(x):=\tau(g(x))$. Since $\mathcal{G}_x \times 1$ keeps all corners invariant, then $\mathcal{G}_x \ltimes R^n_k$ is equivalent to $(\mathcal{G}_x \times Z_2^{n-k}) \ltimes R^n$ with the constrain that $\mathcal{G}_x \times 1$ keeps $R^n_k$ invariant. The strata expands accordingly. By the standard fact (cf. \cite{dk}) the canonical decomposition on $(\mathcal{G}_x \times Z_2^{n-k}) \ltimes R^n$ is a stratification, thus locally finite. Back to $\mathcal{G}_x \ltimes W$, locally the stratas on it are the intersections of these finite stratas with all corners (including the interior of $W$ or its 0-corner) of $W$, which is certainly finite.

Now we verify the condition of frontier.  Again, based on the double argument above we see that the condition of frontier holds on $R^n$. It will then be with no hard to see that, when restricting back to $R^n_k$, the condition of frontier also holds: if a strata in $R^n$ lies completely in $R^n_k$, then its closure still lies in $R^n_k$ and since the isotropy group of each point on $k$-corner is the quotient of the corresponding isotropy group of the same point lying on the "double" $R^n$ by its subgroup $Z^k$ (so that different stratas on $R^n$ corresponds to different stratas on $R^n_k$), our condition of frontiner holds on $R^n_k$ for this strata. Otherwise, a strata crosses some $l$-corners, $1 \le l \le k$. But due to the same reason for isotropy groups of points on that $l$-corners, we still find that the condition of frontier holds for such strata. Note that although different stratas on different corners with different dimensions can in principle glue to form a larger strata (because their isotropy groups after quotienting by different $Z^t$'s, $0 \le t \le k$, may be isomorphic), but by definition of stratification on orbifolds with corners (\eqref{str_mani_cor}) a strata can not cross different corners, so this situation won't happen. 

The continuity of the canonical projection $\pi: M \rightarrow X$ ensures that the condition of frontier also holds on the orbit space.

To prove that it is a Whitney stratification on $M$, we observe that if $(R, S)$ are two strata on the same $k$-corner of $M$, then \cite{cm} shows that Whitney conditiond are true. If $(R, S)$ lie on a $k$-corner of $M$ and its boundary $l$-corner respectively, then for any $x_n\in R$ converging to $x\in S$, and the tangent spaces $T_{x_n}(R)$ converging to $\tau \subset T_x(M)$ (locally around $x$, $T_x(M)$ is given by doubling the local model $R^n_k$, so $T_x(M)$ is $R$-dimensional), $T_x(S)$ must lie in $\tau$, due to the very structure of manifold with corners (i.e. the structure of $R^n_k$). The same reasoning applies for proving that the Whitney condition (B) is also satisfied for such $(R, S)$.

The statement of Morita invariance on orbit space follows from the fact that Morita equivalence relation is invariant on any orbit $O$ on any $k$-corner of $M$.
\end{proof}
We will prove that the above stratification is also a Whitney stratification on $X$. To achieve that, we need a
"corner" case extension of a result of Bierstone on the stratification of orbit space of a proper action groupoid.
\begin{proposition}\label{Whitney_orbit_proper_action_groupoid_corner}
Let $\mathcal{G} \ltimes M$ be a proper action groupoid with corners
, then the orbit space $X:=|M/\mathcal{G}|$ has a Whitney stratification, and it coincides with the canonical stratification induced by Morita type equivalence classes.
\end{proposition}
\begin{proof}
It was shown in the proposition \ref{stratification_proper_groupoid_corner} that the partition by Morita type equivalence classes on $X$ is a stratification. Now we prove it is a Whitney stratification. As this is a local property, by the local model \ref{local model of proper Lie groupoids with corners about a point},  the orbit space $X$ is locally homeomorphism to the orbit space of a linear corner representation of a compact Lie group $\mathcal{G}_x$. Consider, now, $\mathcal{G}_x$ acting linear on $R^n_k$ compatible with the corner structure. According to \cite{b75}, the orbit space of any linear action of a compact Lie group $\mathcal{H}$ on $R^n$ can be realized as a semi-algebraic subset (for semi-algebraic set see \cite{l}) of an affine variety that is induced by the map:
\begin{equation}
\Phi: R^n \rightarrow R^h
\end{equation}
where $\Phi = (\phi_1,\dots,\phi_h)$ is a set of generator of $\mathcal{H}$ invariant polynomials on $R^n$, and can be assumed to be homogeneous. $R^n/\mathcal{H}$ is homeomorphic to the image $\Phi(R^n)$ that is a semi-algebraic subset of the affine variety $V(I)$ where $I$ is the idea of algebraic relations among $\phi_1,\phi_2,\dots,\phi_h$. Using the doubling argument, we expand along $n-k$ axes in $R^n_k$ to turn the latter symmetrically into $R^n$, with the action of $\mathcal{G}_x \times Z_2^{n-k}$ of which the second factor acting by reflection along the orthogonal hyperplane to the corresponding axis, as before. Since $\mathcal{G}_x \times 1$ keeps all corners invariant, $\mathcal{G}_x \ltimes R^n_k$ is equivalent to $(\mathcal{G}_x \times Z_2^{n-k}) \ltimes R^n$ with the constrain that $\mathcal{G}_x \times 1$ keeps $R^n_k$ invariant. Then, as in \cite{b75} we can use $\mathcal{G}_x \times Z_2^{n-k}$ to find a mapping $\Phi=(\phi_1,\phi_2,\dots,\phi_l): R^n \rightarrow R^l$ that induces a homeomorphism from $R^n/(\mathcal{G}_x \times Z_2^{n-k})$ to a semi-algebraic subset of an affine variety in $R^l$. 
As the orbit space of $(\mathcal{G}_x \times Z_2^{n-k}) \ltimes R^n$ is the same as the one of $\mathcal{G}_x \ltimes R^n_k$, restricting to $R^n_k$, we see that $\Phi$ induces a homeomorphism from $R^n_k/\mathcal{G}_x$ to that semi-algebraic set in $R^l$.  Since by the result of \cite{b75} this image on $R^l$ has a Whitney stratification that coincides with the canonical stratification induced by $(\mathcal{G}_x \times Z_2^{n-k}) \ltimes R^n$. Although one needs  to form additional stratas on the orbit of corners of $R^n_k$ by intersecting some canonical stratas with the corresponding orbit of that corners (see \eqref{str_mani_cor}), it is obvious that the Whitney conditions (A) (thus (B)) is true even in this setting (as the tangent space on any point of a strata on the orbit of corners is contained in the tangent space on the same point of the corresponding original canonical strata) . We conclude the proof. 
\end{proof} 

We can now give a "corner" case result corresponding to theorem \ref{Whitney_stratification_proper_groupoid}.
\begin{proposition}\label{Whitney_stratification_orbit_proper_groupoid_corner}
Let $\mathcal{G} \Rightarrow M$ be a proper Lie groupoid with corners, then $M$ and the orbit space $X:=|M/\mathcal{G}|$ has a Whitney stratification, and it coincides with the canonical stratification induced by the Morita type equivalence classes.

Any Morita equivalence between two proper Lie groupoids with corners induces a homeomorphism between their orbit spaces that preserves the canonical stratifications.
\end{proposition}
\begin{proof}
Let $n=dim M$. It was already shown in the proposition \ref{stratification_proper_groupoid_corner} that the canonical decompositions on $M$ and $X$ are stratifications and it is a Whitney stratification on $M$. It is a local problem to verify it is a Whitney stratification on $X$. By the local model \ref{local model of proper Lie groupoids with corners about a point} of $\mathcal{G}$ around a point $x \in M$, we know that it is enough to check this property for a proper action groupoid with corners $\mathcal{G}_x \ltimes U$ where $U \subset N_x \cong R^n_k$ for some $k$. The proposition \ref{Whitney_orbit_proper_action_groupoid_corner} already shows that this is true. The Morita invariance on $X$ is also established in proposition \ref{stratification_proper_groupoid_corner}. 
\end{proof}


\section{Main theorem}
Now we prove our main theorem \ref{main_thm}. Before that we take this opportunity to briefly mention a former effort of us to tackle this problem. We tried to use the global quotient theorem \ref{global_quotient} and a result of Yang (\cite{y}) on the triangulability of the orbit space of transformation groups. However, unfortunately we finally found that his result is incorrect due to using a result of S. S. Cairns (\cite{css}) on the triangulability of so called regular locally polyhedral spaces that is known to be false now. And the global quotient construction on orbifolds with corners is also not very pleasant.
\begin{proof}[Proof of Theorem \ref{main_thm}]
Let $\mathcal{G}$ be a proper Lie groupoid with corners over $M$, being the orbifold structure of an orbifold with corners $P$ in our geometric chains. Proposition \ref{stratification_proper_groupoid_corner} shows that the canonical stratification (induced by the Morita equivalence \ref{morita_corners}) on the orbit space $P \cong M/\mathcal{G}$ is a stratification. And the proposition \ref{Whitney_stratification_orbit_proper_groupoid_corner} shows that it is further a Whitney stratification. 
A classical result of Thom and Marden (cf. \cite{mat70}) shows that any Hausdorff paracompact topological space embedded in a manifold without boundary and with a Whitney stratification admits an abstract pre-stratification structure. The notion of abstract pre-stratification is an axiomatization of the notion of \textit{control data} in \cite{mat70}. See appendix for a precise definition. Although our orbit space $P$ does not necessary be embedded in a manifold, however, locally it did, as from the proof on the proposition \ref{Whitney_orbit_proper_action_groupoid_corner}. One can also refer to section 3 of \cite{ppt} that establishes the pre-stratification structure in the context without ambient manifold\footnote{We realized that \cite{ppt}  follows essentially the same line as ours to prove the triangulability of smooth orbifolds after we complete this work.}. So it indeed admits an abstract pre-stratification structure. A result of Verona (cf. \cite{v}) then shows that any Hausdorff paracompact topological space with an abstract pre-stratification structure (in particular it has a stratification) admits a triangulation subordinate to the pre-stratification (the stratification underlying it). By a triangulation of a topological space $X$ we mean a homeomorphism $tri: X \rightarrow |L|$ from $X$ onto a underlying space of a simplicial complex $L$, and we say that this triangulation is subordinate to a stratification $\mathcal{S}$ of $X$ if for any strata $R \in \mathcal{S}$, $tri|_R$ is the underlying space of a simplicial subcomplex of $L$. Note that the canonical stratification is compatible with the corner structure of $P$.

There is a canonical homomorphism from singular chain complex to geometric chain complex as every simplex is a manifold with corners. We denote the induced homomorphism on homology by
\begin{equation}
p: H_*(X, Q) \rightarrow H^{geom}_*(X, Q)
\end{equation} 
As  we have proved that every orbifold with corners admits a triangulation compatible with the corner structure before, it implies that every geometric chain is homologous to a singular simplicial chain. From this one can deduce that $p$ is an isomorphism as follows:
\begin{itemize}
\item surjectiveness: follows by the fact that every geometric chain is homologous to a singular simplicial chain and the equivalence relation \eqref{eqiv}.
\item injectiveness: let $gc$ be a geometric chain of dimension $n$ that is a boundary of another geometric chain $q$. Suppose there is a $sc \in H_{n}*(X, Q)$ such that $p(sc) = [gc]$ and $sc$ is not a boundary. Choose any triangulation of $q$, it induces a triangulation $\mathcal{T}$ of $\partial q = gc$. Let us denote the simplicial complex associated to $\mathcal{T}$ by $t$. So $sc$ and $t$ give two triangulations of $gc$. Remember that these two triangulations are compatible with corner structure of $gc$, it means $sc$ and $t$ are homologous. So $sc$ is a boundary, thus it is zero in $H_{dim gc}*(X, Q)$.
\end{itemize}

\end{proof}
\begin{proof}[Proof of Corollary \ref{main_cor}]
The proof follows similarly as the preceding proof, simply noting that a simplex is a manifold with corners.
\end{proof}

As stated in the introduction, the geometric chains behavior nicely with some natural operations. We list two of them.
\begin{enumerate}
\item (Pullback).
Let $X, Y$ be topological spaces and $f: Y \rightarrow X$ be a fiber bundle with fiber $F$ a smooth connected oriented manifold possibly with corners (e.g. $S^n$), it induces a pullback $f^*: GC_*(X,Q) \rightarrow GC_*(Y,Q)$ as usual.
\begin{definition}
For $(x: P \rightarrow X) \in GC_*(X,Q)$, $f^*(x): R \xrightarrow{y} Y$ is an element of $GC_*(Y,Q)$ such that the following diagram commutes.
\begin{equation*}
\xymatrix{
  Z \ar[ddr] \ar[rrd] \ar@{.>}[dr] & & \\
  & R \ar[r]\ar[d]^{y} & P \ar[d]^{x} \\
  & Y \ar[r]^f & X
  }
\end{equation*}
and it satisfies the universal property: if $Z$ is another topological space satisfying the same property, then there is a unique morphism $Z \rightarrow R$ making the above diagram commutes.
\end{definition}
By universality, $f^*(x)$ is unique up to a unique isomorphism.
The \textit{existence} can be deduced as follows: $f^*(x)$ exists as a topological space. It is enough to check  locally, in which case it is $Y=X \times F$ with $f$ a projection onto the first factor. 
In this case $f^*(x)$ is isomorphic to $P\times F \xrightarrow{(f,id)} Y$ where $P\times F$ is a smooth connected oriented orbifold with corners. 
\item (Cartesian product).
Let $X,Y$ be two topological spaces, then there is a functor $prod: GC_*(X) \times GC_*(Y) \rightarrow GC_*(X\times Y,Q)$, defined as follows.
\begin{definition}
if $P$ and $Q$ are two smooth connected orbifolds with corners, then $prod((f_X,f_Y)):= (f_X \times f_Y)(R)$ where $R$ is the cartesian product of orbifolds with corners $P$ and $Q$ with the induced orientation, which is also an orbifold with corners.
\end{definition}
\end{enumerate}

As stated in the section of introduction, we expect our GHT has applications in some situations where if using singular chains there will result in some difficulties. In the following section, we demonstrate this by using it in the author's work on Batalin-Vilkovsky (hereafter BV for short) algebraic structures in open-closed topological conformal field theory.

\section{Applications}
As was briefly mentioned in the introduction, we find that the geometric chains is quite suited for the construction of BV structure in the compactified moduli space of Riemann surfaces with boundary and marked points, extending the work of Costello (\cite{cos04,cos05}) for the algebraic definitions of Gromov-Witten potentials. Let us give a bit more details on this construction. Costello originally tries to construct a BV structure on the singular chains on compactified moduli spaces of Riemann surface with marked points. Unfortunately, the BV $\Delta$ operator is defined via pull backs of chains along smooth fiberations, where it doesn't exist for singular chains. He circumvented this problem by using homotopy coinvariants with respect to $S^1$ actions on singular chains (see footprint on page 5 of \cite{cos05}). Equivariant homotopy constructions make $\Delta$ well defined then, it certainly requires much more efforts though (see, for example, section 5 of \cite{cos05}).

An alternative way to define this BV structure found by us is to use geometric chains instead of singular chains. It is quite simpler than tools in $S^1$ equivariant homotopy theory and is essentially elementary. A better thing is that the main theorem \eqref{main_thm} ensures that the further detailed quantitative analysis in \cite{cos05} can be essentially extended directly to open-closed cases. Let us explain this a bit more. 
The following statements are on the BV structure on the space of geometric chains mentioned above.

\begin{definition}
Let $V$ be a graded linear space over field $k$.
A \textit{dg-BV algebraic structure} on $V$ is a
quadruple $(V, \bullet, d, \delta)$, satisfying the following three conditions:
\begin{enumerate}
\item $(V,\bullet,d)$ is a differential, graded, (graded)commutative, (graded)associative algebra
over $k$. The differential $d$ is of degree 1 and $d(1) = 0$.
\item $\Delta$ is a second order differential operator with respect to $\bullet$, i.e. the degree of $\Delta$ is 1,
$\Delta^2 = 0$, $\Delta(1) = 0$, and for any given $a,b,c \in V$,
\begin{align*}
\Delta (a \bullet b \bullet c) & = & \Delta (a \bullet b) \bullet c + (-1)^{|a|}a \bullet \Delta (b \bullet c) + (-1)^{(|a|+1)|b|}b \bullet \Delta (a \bullet c) \\
& &-(\Delta a) \bullet b \bullet c - (-1)^{|a|} a \bullet (\Delta b) \bullet c - (-1)^{|a|+|b|} a \bullet b \bullet (\Delta c)
 \end{align*}
where $||$ is the degree of an element.
\item graded commutator $[d, \Delta] = d \Delta + \Delta d = 0$.
\end{enumerate}
\end{definition}
If $V$ is concentrated in degree 0 (with $d$ trivial), then the dg-BV algebraic structure is a \textit{BV algebraic structure}.
\begin{remark}
Condition 2 is equivalent to the fact that the deviation of the derivative $\Delta$ from being derivation, which is defined by
\begin{equation*}
\{, \} := \Delta (ab) - \Delta (a)b - (-1)^{|a|}a\Delta (b)
\end{equation*}
is a (graded)Lie bracket and $\{, \}$ is a (graded)derivation for each variables,i.e.
\begin{eqnarray*}
\{a,bc\} = b\{a,c\} + (-1)^{|b|}
\{a,b\}c \\9
\{ab,c\} = a\{b,c\} + (-1)^{|a|}
\{a,c\}b
\end{eqnarray*}
The condition 3 is equivalent to d being a (graded)derivation for the Lie bracket{, }, i.e.
\begin{equation*}
d\{a,b\} = \{da,b\} + (-1)^{|a|}
\{a,db\}
\end{equation*}
\end{remark}
The open-closed analogy of the moduli spaces of Riemann surfaces with marked points is the moduli spaces of boarded Riemann surfaces with marked points lying both on the interior and boundaries. More precisely, consider the moduli spaces of isomorphism classes of stable
bordered Riemann surfaces of type $(g,b)$ with $(n, \vec{m})$ marked points and certain extra data,
namely, decorations by a real tangent direction, i.e., a ray, in the complex tensor product of the tangent spaces on each side of each interior node. We denote it by $\underline{M}^{b,\vec{m}}_{g,n}$. We will be working on the moduli spaces:
\begin{equation*}
\underline{M}^{b,m}_{g,n} / \varrho = (\prod_{\vec{m}:\sum {m_i} = m}
\underline{M}^{b,\vec{m}}_{g,n} / Z_{m_1} \times \cdots \times Z_{m_b} ) / \varrho_b \times \varrho_n
\end{equation*}
of stable bordered Riemann surfaces as above with unlabelled boundary components
and marked points, that is, the quotient of the disjoint union
$ \underline{M}^{b,m}_{g,n} / \varrho = \prod_{\vec{m}:\sum {m_i} = m}
\underline{M}^{b,\vec{m}}_{g,n} $
of moduli spaces with labelled boundaries and marked points by an appropriate action
of the permutation group $\varrho = (\prod_{\vec{m}:\sum_{m_i} = m} Z_{m_1} \times \cdots \times Z_{m_b} ) \times \varrho_b \times \varrho_n $. $\underline{M}^{b,\vec{m}}_{g,n}$ and $\underline{M}^{b,m}_{g,n} / \varrho$ are both oriented orbifolds with corners with dimension equals $ 6g + 6 + 2n + 3b + m$. The latter is equipped with a local system $\mathbb{Q}^\epsilon$ coming from a certain sign representation of $\varrho$ on a one-dimensional vector space $L$ over $\mathbb{Q}$, which is the orientation sheaf of $\underline{M}^{b,m}_{g,n} / \varrho$. $\mathbb{Q}^\epsilon$ is the orbifoldic construction from the trivial determinant bundle on $\underline{M}^{b,m}_{g,n}$ (as it is oriented) acted by $\varrho$. See \cite{vhz} for details.

The space we will introduce a BV algebraic structure is the space of geometric chain complexes.
\begin{equation*}
V := \bigoplus_{b,m,n}C^{geom}_{*}(\underline{M}^{b,m}_n / \varrho; \mathbb{Q}^\epsilon)
\end{equation*}
where $\underline{M}^{b,m}_n /\varrho$ is the moduli space of real compactification of stable bordered Riemann surfaces with $b$ boundary
components, $n$ interior marked points, and $m$ boundary marked points, just like the Riemann
surfaces in $\underline{M}^{b,m}_{g,n} / \varrho$, but in general having multiple connected components of various
genera.

One can canonically define an operator $\Delta = \Delta_c + \Delta_o + \Delta_{co}$ on $V$. It is proved that this is a BV operator.
\begin{theorem}
The operator $\Delta = \Delta_c+\Delta_o+\Delta_{co}$ is a graded second-order
differential on the dg graded commutative algebra $V$ and thereby defines the structure of
a dg BV algebra on $V$ .
\end{theorem}
We will not present a detail explanation of the construction of $\Delta$, instead, we refer the interested readers to \cite{vhz} for a very clear account. We just want to emphasize that the BV operator is more naturally and easily constructed in the space of geometric chain complexes than the space of singular chain complexes. This is also a reason we prefer using geometric chains in this case.

This nice result is essentially not completely new, as it is the mathematical rigorous presentation of the ideas developed originally in physics by physicist Sen and Zwiebach. For details, see (\cite{sz}) for closed case and \cite{zwi} for open-closed case.

To develop a pure algebraic treatment of Gromov-Witten invariants, one needs first replace the fundamental classes of the compactified moduli spaces of Riemann surfaces with boundary and marked points by a certain solution of Quantum master equation associated with the BV algebra constructed above. The fundamental result is the following existence and uniqueness of solutions of the associated quantum master equation. 

Let $V^{b,m}_{g,n}$ be the image of the inclusion of $C^{geom}_*(\underline{M}^{b,m}_{g,n}/\varrho; \mathbb{Q}^\epsilon) \rightarrow V$.
\begin{theorem}
For each $g,n,m,b$ with $2 - 2g - n - b - m/2 < -\frac{1}{2}$
, there exists
an element $S^{b,m}_{g,n} \in V^{b,m}_{g,n}$ of degree 0, with the following properties.
\begin{enumerate}
\item $S^{0,0}_{0,3}$
is a 0-chain in the moduli space of Riemann spheres with 3 unparameterised,
unordered closed boundaries and with no open or free boundaries.
\item Form the generating function
\begin{equation*}
S = \sum_{\substack{g,n,b,m \ge 0 \\ 2g+n+b+m/2-2>0}}\hbar^{p-\chi}\lambda^{-2\chi}S^{b,m}_{g,n} \in \lambda F(M)[[\sqrt{\hbar},\lambda]]
\end{equation*}
here $p = 1 - \frac{m+n}{2},\chi = 2 - 2g - n - b - \frac{m}{2}$.
\end{enumerate}
S satisfies the Batalin-Vilkovisky quantum master equation:
\begin{equation*}
\hat{d}e^{S/\hbar} = 0
\end{equation*}
(Remember that in dg BV algebra $B[[\hbar,\lambda]]$, we let $\hat{d} = d + \hbar \Delta$).
Equivalently,
\begin{equation*}
\hat{d}S + \frac{1}{2}\{S, S\} = 0
\end{equation*}
Further, such an $S$ is unique up to homotopy through such elements.
\end{theorem}
This result has the same form as the one for closed case. By our main theorem, the proof is essentially identical to the proof in \cite{cos05}. For more details see \cite{hy}.

As for closed cases, this solution also encodes the fundamental chains of moduli spaces of Riemann surfaces with boundary and marked points, as demonstrated in \cite{vhz}:
\begin{theorem}
For the notation of above theorem, if we take $S^{b,m}_{g,n}:= [\underline{M}^{b,m}_{g,n}/\varrho] \in C^{geom}_0(\underline{M}^{b,m}_{g,n}/ \varrho; \mathbb{Q^\epsilon})$ be the fundamental (geometric) chains. Then $S$ satisfies the quantum master equation:
\begin{equation*}
(d + \hbar \Delta)e^{S/\hbar} = 0
\end{equation*}
\end{theorem}

Together with results in \cite{cos05}, these theorems essentially state that the fundamental classes of moduli spaces of Riemann surfaces with marked points, or the (geometric) fundamental chains of moduli spaces of boarded Riemann surfaces with marked points, are encoded in the unique (up to homotopy) solution of the quantum master equation associated with a canonically defined BV algebraic structure on the singular (or geometric) chain complexes. From this point, one can proceed to define Gromov-Witten potential for any (open-closed) topological conformal field theory, a subject called categorical GW theory initiated by Costello, motivated by the homological picture of mirror symmetry of Kontesvich. The closed case was completed, the open-closed case is still open, I guess.


Finally, we remark that the theory in this paper only holds for \textit{differentiable} orbifolds with corners. We don't expect similar results hold in topological category, i.e. topological orbifolds with corners (in whatever sense), at least these can not be deduced using methods in this paper, since C.Manolescu (\cite{man}) showed that there is a topological manifold in each dimension $\geq 5$ that is not triangulable. We will leave it as an interesting topic for future study.

We believe the geometric homology theory developed in this paper should be of many applications in relevant areas due to it is a very flexible and straightforward generalization of singular simplicial homology theory and has many advantages overcoming the unpleasant combinatoric rigidity of simplexes. In the future we plan to discuss if it can be useful in developing chain-level Gromov-Witten theory, as an example.

\section{Appendix}
\appendix
A closely related notion to the notion of Whitney stratification is {abstract pre-stratification}. It is a very essential concept in the stratification and triangulation theory (cf. \cite{mat70, v}).
\begin{definition}\label{abstract_stratification}
An \textit{abstract pre-stratified} set is a triple $(V, \mathcal{S}, \mathcal{J})$ satisfying the following axioms.
\begin{enumerate}
\item $V$ is as Hausdorff, locally compact topological space with a countable basis for its topology.
\item $\mathcal{S}$ is a family of locally closed subsets of $V$, such that $V$ is the disjoint union of the members of.
The members of $\mathcal{S}$ will be called the strata of $V$.
\item Each stratum of $V$ is a topological manifold (in the induced topology), provided with a
smoothness structure.
\item The family $\mathcal{S}$ is locally finite.
 \item The family $\mathcal{S}$ satisfies the axiom of the frontier: if $X, Y$ and $Y \cap \bar{X} \neq \emptyset$, then $Y \subseteq \bar{X}$.
If $Y \subseteq \bar{X}$ and $Y \neq {X}$, we write $Y < X$. This relation is obviously transitive: $Z < Y$ and $Y < X$ imply
$Z < X$.
 \item $\mathcal{J}$ is a triple $\{(T_X), (\pi_X), (\rho_X)\}$, where for each $X \in \mathcal{S}$, $T_X$ is an open neighbourhood of $X$ in $V$, $\pi_X$ is a continuous retraction $T_X$ of onto $X$, and $\rho_X : X \rightarrow [0, \infty)$ is a continuous function.
We will $T_X$ call the tubular neighbourhood of $X$ (with respect to the given structure of a prestratified set on $V$), $\pi_X$ the local restriction of $T_X$ onto $X$ and $\rho_X$ the tubular function of $X$.
\item $X = \{v \in T_X : \rho_X(v) = 0\}$.
If $X$ and $Y$ are any strata, we let $T_{X, Y} =T_X \cap Y$, $\pi_{X, Y} = \pi_X | T_{X, Y}$, and $\rho_{X, Y} = \rho_X | T_{X, Y}$,. Then $\pi_{X, Y}$ is a
mapping of $T_{X, Y}$ into $X$ and $\rho_{X, Y}$ is a mapping $\rho_{X, Y}$ into $(0, \infty)$. Of course, $T_{X, Y}$ may be empty, in
which case these are the empty mappings.
\item For any strata $X$ and $Y$ the mapping
$(\pi_{X, Y}, \rho_{X, Y}) : T_{X, Y} \rightarrow X \times (0, \infty)$
is a smooth submersion. This implies dim $X <$ dim $Y$ when $T_{X, Y} \neq \emptyset$.
\item For any strata $X, Y$, and $Z$, we have
$\pi_{X,Y}\pi_{Y,Z}(v) = \pi_{X, Z}(v)$ \\
$\rho_{X, Y}\pi_{Y, Z}(v) = \rho_{X, Z}(v)$
whenever both sides of this equation are satisfied, i.e., whenever $v \in T_{Y, Z}$ and $\pi_{Y, Z}(v) \in T_{X, Y}$.
\end{enumerate}
\end{definition}

\noindent  Hao Yu, {\it School of Mathematics, University of Minnesota, Minneapolis, MN 55455, USA, and\\
\indent AI center of Cyclone, Chaoyang District, Beijing 100000, China.}
\\
{\tt dustless2014@163.com}

\end{document}